\documentclass[12pt]{amsart}
\usepackage{psfig}
\date{}
\newtheorem{theorem}{Theorem}
\newtheorem{proposition}{Proposition}

\theoremstyle{definition}

\newcommand{\IR}{{\mathbb{R}}}

\newcommand{\bs}{\backslash}

\newcommand{\0}{\emptyset}
\newcommand{\inte}{\mathop{\rm int}\nolimits}
\newcommand{\conv}{\mathop{\rm conv}\nolimits}
\newcommand{\diam}{\mathop{\rm diam}\nolimits}

\newcommand{\eps}{\varepsilon}
\renewcommand{\phi}{\varphi}

\begin{document}
\title{On curves contained in convex subsets
of the plane}
\author{Don Coppersmith}
\author{Gyozo Nagy}
\author{Sasha Ravsky}
\email{dcopper@us.ibm.com;gyozo.nagy@raiffeisen.hu;oravsky@mail.ru}
\address{Department of Mathematical Sciences,
IBM Thomas J Watson Research Center,
Yorktown Heights NY 10598, USA;
Technical Center, Raiffeisen Hungary, Budapest Kesmark 11-13, 1158, Hungary;
Department of Mathematics, Ivan Franko Lviv National University,
Universytetska 1, Lviv, 79000, Ukraine}
\begin{abstract}
If $K'\subset K$ are convex bodies of the plane
then the perimeter of $K'$ is not greater than the perimeter of $K$.
We obtain the following generalization of this fact.
Let $K$ be a convex compact body
of the plane with the perimeter $p$ and the
diameter $d$ and $r>1$ be an integer. Let $s$ be the
smallest number such that for any curve of length
greater than $s$ contained in $K$ there is a straight line intersecting
the curve at
least in $r+1$ different points. Then $s=rp/2$ if $r$ is even and
$s=(r-1)p/2+d$ if $r$ is odd.
\end{abstract}

\maketitle \baselineskip15pt

Let $K$ be a compact subset of the plane.  Recall that the set $K$ is
called {\em convex} if for each pair of points $x,y\in K$ the segment
$[x,y]$ connecting these points belongs to $K$.  The set $K$ is called
{\em strictly convex} if for each pair of points $x,y\in K$ the
interval $(x,y)=[x,y]\bs\{x,y\}$ belongs to the interior $\inte K$ of
the set $K$.  Equivalently, the set $K$ is strictly convex if its
boundary $dK$ contains no segments of nonzero length.

Let $I=[0;1]$ be the unit segment.  A {\em curve} is a continuous
map $\phi:I\to\IR^2$. Let $0=a_0<\dots<a_n=1$ be a sequence. Let
$[\varphi]$ be the broken line with the sequence of vertices
$\varphi(a_0),\dots,\varphi(a_n)$ that is the union of the segments
$\bigcup_{i=0}^{n-1}[\varphi_i;\varphi_{i+1}]$. By
$l(a_0,\dots,a_n)$ we denote the length of the broken line $[\varphi]$.
We say that the curve $\phi$ has {\em length}
$l(\phi)$ if for each $\eps>0$ there is
$\delta>0$ such that $|l(\phi)-l(a_0,\dots,a_n)|<\eps$ for any
sequence $0=a_0<\dots<a_n=1$ with $a_{i+1}-a_i<\delta$ for each $0\le
i\le n-1$.

The convex subset $K$ of the plane is called a {\em convex body} if
$\inte K\not=\0$.  If $K$ is a convex body then the boundary $dK$ of
$K$ is an image of a curve and there exists the length $l(dK)$ which
is equal to the upper bound of the perimeters of convex polygons
\footnote{We can consider a convex polygon as the convex hull of a
broken line}
inscribed in $K$ \cite[p.  373]{Ale},\cite[p.28]{YB}.  The length of
the boundary we shall call the {\em perimeter} of the body $K$ and
shall denote as $p(K)$.

It is well known that $p(K')\le p(K)$ for every pair of convex bodies
$K$ and $K'$ with $K'\subset K$ (\cite[p.  373]{Ale},\cite[p.28]{YB}).
This means that if the image of a curve $\phi:I\to K$ is a boundary of
a convex body then $l(\phi)\le p(K)$.  A slightly weaker claim is: if
the image of a curve $\phi:I\to K$ is a boundary of a strictly convex
body then $l(\phi)\le p(K)$.  The boundary $dK'$ of the strictly
convex body $K'$ has the property that every line intersects $dK'$ in
at most two points.  We are going to generalize the claim showing that
$l(\phi)\le p(K)$ for every curve $\phi:I\to K$ with every line
intersecting the image of $\phi$ at most in two points.  This result
is closely related with the following proposition.

\begin{proposition} Let $S^1$ be a circle and $\phi:S^1\to\IR^2$ be a
homeomorphic embedding. Then the following conditions are equivalent:

(1) $\phi(S^1)$ is a boundary of a strictly convex body.

(2) Every line intersects the image of $\phi$ at most in two points.

(3) Every line intersects the image of $\phi$ at most in three points.
\end{proposition}
\begin{proof} The implication $(1)\Rightarrow (2)$ follows from the
definition of the
strictly convex body and the implication $(2)\Rightarrow (3)$ is obvious.
Now we show
the implication $(3)\Rightarrow (1)$. Put $K=\conv\phi(S^1)$.
Since $\phi$ is a
homeomorphism then $K$ is not a segment and hence $K$ is a
convex compact body.
Let $x\in S^1$ be an arbitrary point. Let $C\subset\IR^2$ be the
circle of radius $1$ with the center at the point $\phi(x)$ and
$p:\IR^2\bs\{\phi(x)\}\to C$ be the radial projection.
Since the map $p$ is continuous the set
$p\phi(S^1\bs\{ x\})$ is connected and therefore an arc.
Suppose that the length of $p\phi(S^1)$ is greater than $\pi$.
Then there
exist points $x_1$, $x_2$, $x_3\in S^1\bs\{x\}$ such that the
point $\phi(x)$ lies in
the interior of the triangle with the vertices $p\phi(x_1)$,
$p\phi(x_2)$ and $p\phi(x_3)$.
Then the point $\phi(x)$ lies in the interior of the triangle with the
vertices $\phi(x_1)$, $\phi(x_2)$ and $\phi(x_3)$.  Fix an orientation
on the circle $S^1$.  After a re-enumerating we may suppose that the
points lie on the circle $S^1$ in the order $x$, $x_1$, $x_2$, $x_3$.
Let $l\subset\IR^2$ be a line separating
the points $\phi(x)$ and $\phi(x_2)$ from the points $\phi(x_1)$ and
$\phi(x_3)$. Since the images under the map $\phi$ of the oriented arcs
$(x,x_1)$, $(x_1,x_2)$, $(x_2,x_3)$ and $(x_3,x)$ are linearly connected
then each of the images contain a point from the line $l$. Obtained
contradiction shows that the length of the arc $p\phi(S^1\bs\{x\})$ is
not greater than $\pi$.

Therefore there is a line $l$ going through the point
$\phi(x)$ such that one of the open
half planes created by $l$ contains no points of the set
$p\phi(S^1\bs\{x\})$ and hence no points of the set $\phi(S^1)$.  This
implies that the point $\phi(x)$ is a boundary point of the set $K$.

Hence $\phi(S^1)\subset dK$. Since $ dK$ is homeomorphic to a
circle \cite[p. 372]{Ale} and $\phi(S^1)\subset dK$ is homeomorphic to a
circle the $\phi(S^1)$ must coincide with $dK$. The set $K$ is strictly
convex since its boundary $\phi(S^1)$ contains no segments of nonzero length.
\end{proof}

The main result of the paper is the following

\begin{theorem} Let $K\subset\IR^2$ be a convex compact body
with perimeter $p$, diameter $d$ and let $r>1$ be an integer.
Let $s$ be the smallest number such that for any curve
$\varphi\subset K$ of length
greater than $s$ there is a line intersecting the curve $\varphi$
at least in
$r+1$ different points.
Then
$$s=
\begin{cases}
rp/2, \mbox{if}\; $r$\; \mbox{is even} \\
(r-1)p/2+d,  \mbox{if}\; $r$\; \mbox{is odd.}\\
\end{cases}$$
\end{theorem}
\begin{proof}
Fix the compact body $K$ and the number $r$.
Let $p$ be the perimeter of $K$ and $d$ be the diameter of $K$.

{\em The upper bound}.
Put $s=rp/2$ if $r$ is even and $s=(r-1)p/2+d$ if $r$ is odd. Let
$\varphi:I\to K$ be a curve of length greater than $s$. By the
definition of the length we can choose numbers
$0=a_0<...<a_n=1$ such that
the length of the broken line with the sequence of vertices $
\varphi(a_0),...,\varphi(a_n)$ is greater than $s$.
Let $l_1,\dots,l_n$ be the
lengths of the segments of the broken line and
$\alpha_1,\dots,\alpha_n$ be the
angles with these segments and the $Ox$ axis.
For an angle $\alpha\in [0;2\pi]$,
by $l(\alpha)$ we denote the sum of the lengths of the projections of the
segments $l_i$ onto the line
$L_\alpha=\{(x,y):(x,y)=t(\cos\alpha,\sin\alpha)\}$. Thus $l(\alpha)=\sum
l_i|\cos(\alpha-\alpha_i)|$. By $k(\alpha)$ we denote the length of the
projection of $K$ onto the line $L_\alpha$.
By Cauchy formula \cite[6.1.5]{Had}

$$\int_0^{2\pi}k(\alpha)d\alpha=2p.$$

At first suppose that $r$ is even. To the
rest of the proof it suffice to show that there is an angle $\alpha'$ such that
$l(\alpha')>rk(\alpha')$. Suppose the contrary. Then $$
 2rp=r\int_0^{2\pi}k(\alpha)d\alpha\ge
 \int_0^{2\pi}l(\alpha)d\alpha=
 \int_0^{2\pi}\sum l_i|\cos(\alpha-\alpha_i)| d\alpha=
$$ $$
 \sum\int_0^{2\pi} l_i|\cos(\alpha-\alpha_i)| d\alpha=
 \sum\int_0^{2\pi} l_i|\cos(\beta)| d\beta=
 \sum 4l_i>4s=2rp,
 $$
the contradiction.

Now suppose that $r$ is odd.  Add to the broken line the segment
connecting its ends.  Let the length of the segment be $l_0$ and
$\alpha_0$ be the angle with the segment and the $Ox$ axis.  To the
rest of the proof it suffice to show that there is an angle $\alpha'$
such that
$l(\alpha')>rl_0|\cos(\alpha'-\alpha_0)|+(r-1)(k(\alpha')-l_0|
\cos(\alpha'-\alpha_0)|)=(r-1)k(\alpha')+l_0|\cos(\alpha'-\alpha_0)|$.
That is because each line not intersecting
$l_0$ and the vertices of the broken line intersects the broken line
an even number of times.

Suppose the contrary. Then $$ 2(r-1)p+4l_0=
\int_0^{2\pi}(r-1)k(\alpha)+l_0|\cos(\alpha-\alpha_0)|d\alpha\ge
 \int_0^{2\pi}l(\alpha)d\alpha=
 \int_0^{2\pi}\sum l_i|\cos(\alpha-\alpha_i)| d\alpha=
$$ $$
 \sum\int_0^{2\pi} l_i|\cos(\alpha-\alpha_i)| d\alpha=
 \sum\int_0^{2\pi} l_i|\cos(\beta)| d\beta=
 \sum 4l_i>4s=2(r-1)p+4d,
 $$
which yields the contradiction since $l_0\le d$.\qed
\vskip5pt
\hskip5pt

{\em The lower bound}. The idea of the proof is the following. Put $n=\lfloor
r/2\rfloor$. If $r$ is even, let the curve go $n$ times around the perimeter,
curving slightly to avoid any straight lines but remaining convex. If $r$ is
odd, let the curve go almost $n$ times around the perimeter, then down the
diameter, again always curving slightly. The reason we go ``almost" $n$ times
around the perimeter is to ensure that any straight line intersecting this
``curved diameter" twice, will intersect the perimeter-curve at most $r-2$
times rather than $r-1$. The picture illustrates the construction when $K$ is a
square and $r=5$.

\begin{figure}[h]
\hskip300pt
\psfig{file=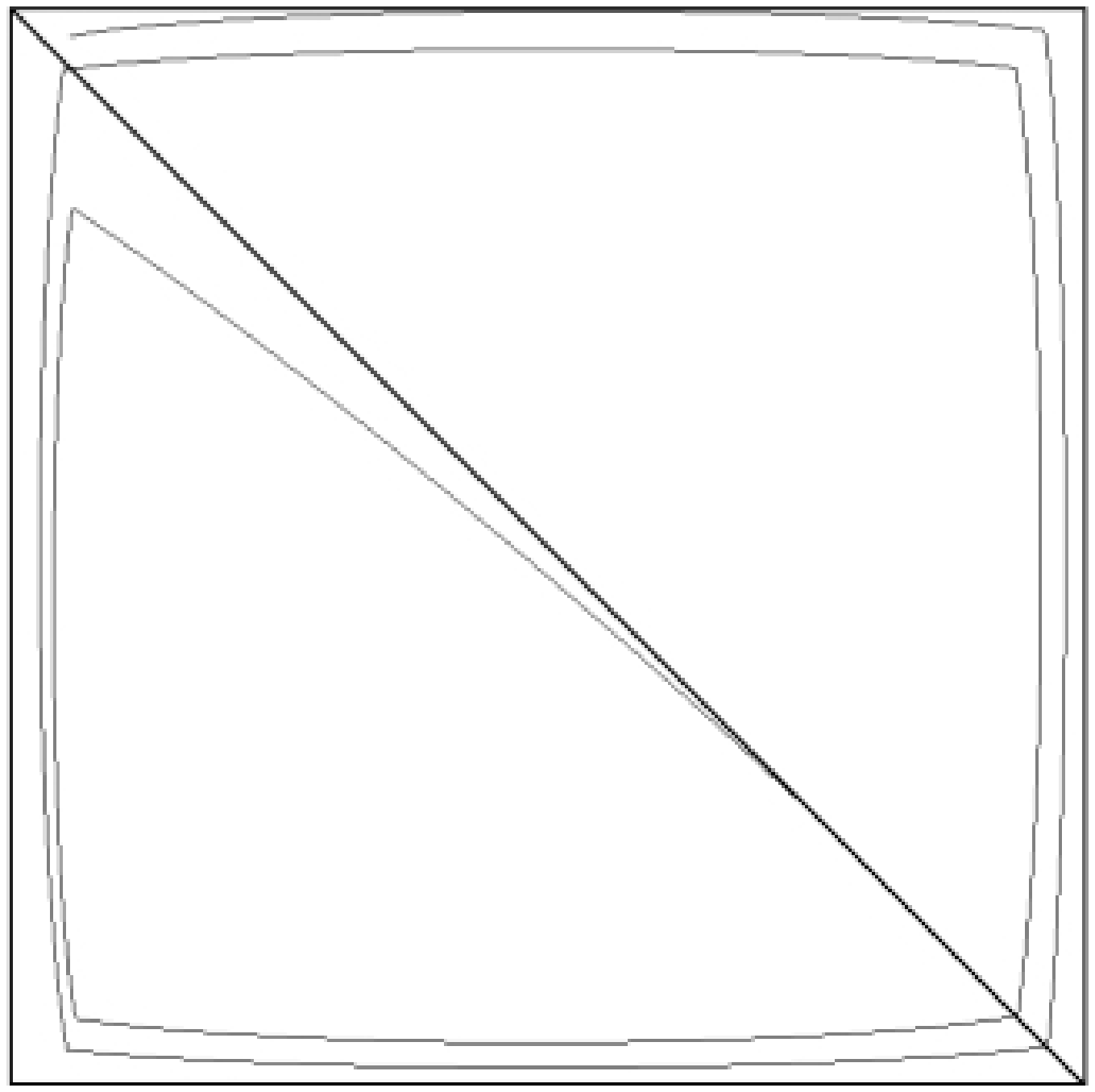,width=100mm}
\end{figure}

Now more precisely. Let $\eps>0$. At first we suppose that $r$ is even.
There is a curve $\phi_1:I\to K$ such that $\phi_1(0)=\phi_1(1)$,
$|l(\phi_1)-p|<\eps$ and
$\phi_1(I)$ is the boundary of a strictly convex body. Choose a number
$\delta_1>0$ such that $l(\phi_1([1-\delta_1;1]))<\eps$. Similarly, there
is a curve $\phi_2:I\to K$ such that $\phi_2(0)=\phi_2(1)$, $|l(\phi_1)-l(\phi_2)|<\eps$,
$\phi_2(I)$ is the boundary of a strictly convex body,
$\phi_2(0)=\phi_1(1-\delta)$
and $\phi_2((0;1))\subset\inte\conv\phi_1(I)$. Choose a number
$\delta_2>0$ such that $l(\phi_2([1-\delta_2;1]))<\eps$.
Similarly to the
previous we can construct the curve $\phi_3:I\to K$ and so on.
From the curves
$\phi_1,\dots,\phi_n$ it is easily to construct the curve $\phi$
intersecting every straight line in at most $r$ points.

Now suppose that $r$ is odd.  To construct the curve $\phi$ we shall
proceed similarly to the case of even $n$.  By induction we can
construct the curves $\phi_1,\dots,\phi_n$ such that

(i) $|\diam(\phi_1(I))-\diam(K)|<\eps$,

(ii) $|\diam(\phi_i(I))-\diam(\phi_{i-1}(I))|<\eps$ for each $i$,

(iii) there is a number $\alpha_i$ such that
$|\diam(\phi_i(I))-|\phi_i(0)-\phi_i(\alpha_i)||<\eps$ for each $i$.

Now choose a
number $\delta_n>0$ such that
$|\diam(\phi_n(I))-|\phi_i(1-\delta_n)-\phi_i(\alpha_i)||<\eps$
and $l(\phi_n([1-\delta_n;1]))<\eps$. There exists a
curve $\phi_{n+1}:I\to K$ such that

(i) $\phi_{n+1}(0)=\phi_n(1-\delta_n)$ and $\phi_{n+1}(1)=\phi_n(\alpha_n)$,

(ii) $\phi_{n+1}((0;1))$ lies in the interior of the triangle
with the vertices
$\phi_n(0)$, $\phi_n(\alpha_n)$ and $\phi_n(1-\delta_n)$,

(iii) $\phi_{n+1}(I)\cup [\phi_n(1-\delta_n);\phi_n(\alpha_n)]$
is the boundary of a convex body,

(iv) no three points of the set $\phi_{n+1}(I)$ lie on a straight line.

From the curves $\phi_1,\dots,\phi_{n+1}$ it is easy to construct
the curve
$\phi$ intersecting every straight line in at most $r$ points.
\end{proof}

\end{document}